\documentclass{amsart}

\usepackage{amsmath}
\usepackage{amscd}
\usepackage{amssymb}

\input xy
\xyoption{all}

\newcommand{\bpsi}{\bar{\psi}}

\newcommand{\bC}{{\mathbb C}}

\newcommand{\bF}{{\mathbb F}}

\newcommand{\bZ}{{\mathbb Z}}

\newcommand{\cB}{{\mathcal B}}
\newcommand{\cC}{{\mathcal C}}

\newcommand{\cM}{{\mathcal M}}
\newcommand{\cO}{{\mathcal O}}

\newcommand{\half}{\frac{1}{2}}

\newcommand{\cW}{{\mathcal W}}
\newcommand{\cX}{{\mathcal X}}
\newcommand{\Mbar}{\overline{\cM}}

\newcommand{\cl}[1]{[\![{#1} ]\!]}
\newcommand{\cor}[1]{\langle {#1} \rangle}

\DeclareMathOperator{\rank}{rank}  \DeclareMathOperator{\ch}{ch}
 
\DeclareMathOperator{\Irr}{Irr} \DeclareMathOperator{\Conj}{Conj}
 \DeclareMathOperator{\Hilb}{Hilb} \DeclareMathOperator{\Li}{Li}

\newtheorem{theorem}{Theorem}[section]
\newtheorem{theorem/definition}{Theorem/Definition}[section]
\newtheorem{Theorem}{Theorem}
\newtheorem{proposition}{Proposition}[section]
\newtheorem{lemma}{Lemma}[section]

\newtheorem{Conjecture}{Conjecture}

\theoremstyle{remark}

\theoremstyle{definition}

\newcommand{\be}{\begin{equation}}
\newcommand{\ee}{\end{equation}}
\newcommand{\bea}{\begin{eqnarray}}
\newcommand{\eea}{\end{eqnarray}}
\newcommand{\ben}{\begin{eqnarray*}}
\newcommand{\een}{\end{eqnarray*}}
\newcommand{\bet}{\begin{equation}
\begin{split}}
\newcommand{\eet}{\end{split}
\end{equation}}

\begin{document}

\title[Crepant Resolution Conjecture in All Genera]
{Crepant Resolution Conjecture in All Genera for Type A  Singularities}
\author{Jian Zhou}
\address{Department of Mathematical Sciences\\Tsinghua University\\Beijing, 100084, China}
\email{jzhou@math.tsinghua.edu.cn}

\begin{abstract}
We prove an all genera version of the Crepant Resolution Conjecture of Ruan and Bryan-Graber for
type A surface singularities.
We are based on a method that explicitly computes Hurwitz-Hodge integrals described in an earlier paper
and some recent results by Liu-Xu for some intersection numbers on the Deligne-Mumford moduli spaces.
We also generalize our results to some three-dimensional orbifolds.
\end{abstract}
\maketitle

\section{Introduction}

Let $\cX$ be an orbifold with coarse moduli space $X$ and let $\pi: Y \to X$ be a crepant resolution.
Under the general principle often referred to as the McKay correspondence in the mathematical literature
(suggested also by work in string theory literature on orbifolds, e.g. \cite{Dix-Har-Vaf-Wit}),
it is expected that invariants of $Y$ coincides with suitably defined orbifold invariants of $\cX$.
See e.g. Reid \cite{Rei} for an exposition of some examples of classical invariants,
e.g. Euler numbers,
cohmomology groups, K-theory and derived categories, etc.
People are also interested in quantum invariants such as Gromov-Witten invariants
and expect a quantum McKay correspondence.
Gromov-Witten invariants of smooth varieties have been developed for quite some time.
Orbifold Gromov-Witten invariants have been
developed more recently for symplectic orbifolds by Chen-Ruan \cite{Che-Rua}
and for Deligne-Mumford stacks by Abramovich-Graber-Vistoli \cite{Abr-Gra-Vis}.
Genus zero (orbifold) Gromov-Witten invariants can be used to define (orbifold) quantum cohomology.
A long standing conjecture of Ruan \cite{Rua} states that the small orbifold quantum cohomology of $\cX$
is related to the small quantum cohomology of $Y$ after analytic continuation
and suitable change of variables.
This version of quantum McKay correspondence is referred to as the Crepant Resolution Conjecture
(CRC).
Recently, Bryan and Graber \cite{Bry-Gra} conjectured the explicit formula
for the change of variables in CRC
for $\cX = [V/G]$,
where $V = \bC^2$ or $\bC^3$,
$G \subset SU(2)$ or $SO(3)$ is a finite subgroup (the binary polyhedral group or the polyhedral group),
and $\pi:Y \to V/G$ is the canonical crepant resolution by $G$-Hilbert schemes $G-\Hilb(V)$
(see e.g. \cite{Bri-Kin-Rei}).
In these cases,
both the orbifold and its crepant resolution are noncompact,
but both admit natural $\bC^*$-actions,
and one can define and work with equivariant Gromov-Witten and orbifold Gromov-Witten invariants
respectively.
By \cite{Bri-Kin-Rei},
there is a canonical basis for $H^*_{\bC^*}(Y)$ indexed by $R \in \Irr(G)$,
irreducible representations of $G$;
on the other hand,
there is a canonical basis of $H^*_{\bC^*, orb}([V/G])$ indexed by $\cl{g} \in \Conj(G)$,
conjugacy classes of $G$.
Denote the corresponding cohomological variables by $\{y_R\}_{R \in \Irr(G)}$
and $\{x_{\cl{g}}\}_{\cl{g} \in \Conj(G)}$ respectively.
Let $y_0$ and $x_0$ be the variables corresponding to the trivial representation and the trivial
conjugacy classes respectively.
The conjecture stated in \cite{Bry-Gra} (attributed to Bryan and Gholampour) is

\begin{Conjecture}
In the case of
$\pi: G-\Hilb(V) \to V/G$,
where $G$ is a polyhedral or binary polyhedral group,
the Crepant Resolution Conjecture holds with the change of variables given by
\ben
&& y_0 = x_0, \\
&& y_R = \frac{1}{|G|} \sum_{g\in G} \sqrt{\chi_V(g) - \dim V} \chi_R(g) x_{\cl{g}}, \\
&& q_R = \exp \biggl( \frac{2\pi i \dim R}{|G|} \biggr),
\een
where $R$ runs over the nontrivial irreducible representations of $G$.
\end{Conjecture}

By the classical McKay correspondence \cite{McK, Gon-Ver, Bri-Kin-Rei},
the geometry of $G-\Hilb(V)$ gives rise to a Dynkin diagram of ADE type,
and hence a simply-laced root system.
Denote by $\alpha_1, \dots, \alpha_n$ the simple roots.
For a positive root $\beta\in R^+$,
let
$\beta = \sum_k b_k \alpha_k$.
Let $\sum_k n_k \alpha_k$ be the largest root.
Bryan-Gholampour \cite{Bry-Gho} reformulate the above conjecture as follows:

\begin{Conjecture}
Let $F_{\cX}(x_1, \dots, x_n)$ denote the $\bC^*$-equivariant genus $0$
orbifold Gromov-Witten potential of the orbifold $\cX=[\bC^2/G]$,
where we  have set the unit parameter $x_0$ equal to zero.
Let $R$ be the root system associated to $G$. Then
\be
F_{\cX}(x_1, \dots, x_n) = 2t \sum_{\beta \in R^+} h(Q_{
\beta}),
\ee
where $h(u)$ is a series with
\be
h'''(u) = - \half \tan \big(\frac{u}{2}\big)
\ee
and
\be
Q_\beta = \pi + \sum_{k=1}^n \frac{b_k}{|G|} \biggl(2\pi n_k
+ \sum_{g\in G} \sqrt{2 - \chi_V(g)} \bar{\chi}_k(g) x_{\cl{g}} \biggr),
\ee
where $n_k$ are the coefficient of $\beta$ and $n_k$ are the coefficients of
the largest root.
\end{Conjecture}

It has been shown in \cite{Bry-Gho}
that Conjecture 2 is equivalent to Conjecture 1 in binary polyhedral case plus an explicit formula for
the Gromov-Witten potential function of $G-\Hilb(\bC^2)$.
There are a number of earlier results.
Ruan's Crepant Resolution Conjecture was established for $G=\bZ_2$ and $\bZ_3$ by Perroni \cite{Per}.
Conjecture 2 was proved
for $G=\bZ_2$ by Bryan-Graber \cite{Bry-Gra},
for $\bZ_3$ by Bryan-Graber-Pandharipande \cite{Bry-Gra-Pan},
for $\bZ_4$ by Bryan-Jiang \cite{Bry-Jia}.
The polyhedral version of Conjecture 2 was proved by Bryan-Gholampour \cite{Bry-Gho2}
for $G=A_4$ and $\bZ_2 \times \bZ_2$.
Conjecture 1 was proved for $G$ of type $A$ in a version due to Perroni \cite{Per}
by Coates-Corti-Iritani-Tseng \cite{Coa-Cor-Iri-Tse}
by mirror symmetry from Givental formalism
(see also the related work by Skarke \cite{Ska} and Hosono \cite{Hos}).

Bryan and Graber \cite{Bry-Gra} also conjectured the higher genera version of the Crepant Resolution
Conjecture.
See also
the recent paper by Coates and Ruan \cite{Coa-Rua} from the point of view of Givental's formalism.
Maulik \cite{Mau} has computed the equivariant Gromov-Witten invariants in all genera of
the minimal resolution of $\bC^2/G$, $G \subset SU(2)$.
Therefore,
to establish the CRC in all genera in the binary polyhedral case,
a key step is to compute the equivariant orbifold Gromov-Witten invariants in all genera of
$[\bC^2/G]$.
In  this paper we will carry out the calculations of the stationary part of the potential function
for $G=\bZ_n$.
It is straightforward to see that  our results match with that of Maulik\footnote{
The author thanks Hsian-Hua Tseng  for an email correspondence on January 25, 2008
which informed him this before Maulik's paper was posted on the arxiv.}.

For $k \geq 0$ and $1 \leq a \leq n-1$,
define
\ben
&& y_{a,k} = \frac{2i}{n} \sum_{b=1}^{n-1} \sin \frac{b\pi}{n} \cdot \xi_n^{ab} x_{b,k} u_b.
\een
For $k \geq 0$, $1 \leq s \leq t \leq n-1$, define
\ben
y_{s\to t, k} = \sum_{a=s}^t y_{a,k}.
\een
Our main result is:

\begin{Theorem} \label{thm:Main}
Up to polynomial terms of degree $\leq 3$ in $y_{a, k}$,
the stationary potential function $F_g^{[\bC^2/\bZ_n]}(x_{1,k}, \dots, x_{n-1, k})_{k \geq 0} $
(defined in \S \ref{sec:Potential})
of the equivariant orbifold Gromov-Witten invariants of $[\bC^2/\bZ_n]$ is equal to
the coefficient of $z^g$ of
\ben
&& (-1)^{g} 2t \sum_{d=1}^{\infty} d^{2g-3}\sum_{1 \leq s \leq t \leq n-1} \biggl(\xi_n^{t-s+1}
\exp \biggl( \sum_{k \geq 0} y_{s\to t, k} \frac{z^k}{4^k\cdot (2k+1)!!} \biggr) \biggr)^d
\een
after analytic continuations,
where $\xi_n = e^{2\pi i /n}$.
\end{Theorem}

We prove this result by explicit computations of the relevant Hurwitz-Hodge integrals
by the method described in an earlier paper \cite{Zho}.
The idea is to combine Tseng's GRR relations for Hurwitz-Hodge integrals
with the results of Jarvis-Kimura \cite{Jar-Kim} on $\bpsi$-integrals,
as suggested by Tseng \cite{Tse}.
This follows the strategy of Faber \cite{Fab} in the case of ordinary Hodge integrals
where he combined Mumford's GRR relations \cite{Mum} with the results on $\psi$-integrals
computed by Witten-Kontsevich theorem \cite{Wit, Kon}.
Another key ingredient is the Hurwitz-Hodge version of Mumford's relations \cite{Mum}
as established by Bryan-Graber-Pandharipande \cite{Bry-Gra-Pan}.
This is used to convert the relevant Hurwitz-Hodge integrals to a simple one
which involves only one Chern character of the Hurwitz-Hodge bundles.
For other work on Hurwitz-Hodge integrals,
see e.g. \cite{Fab-Pan, Bry-Gra,Bry-Gra-Pan, Bay-Cad, Bou-Cav, Cad-Cav, Cav1, Cav2, Jar-Kim2}.

In our computation of the Hurwitz-Hodge integrals we encounter
some intersection numbers on the Deligne-Mumford moduli spaces
(see \S \ref{sec:I1} and \S \ref{sec:I2}).
They can be computed using some recent results on the $n$-point functions
of intersection numbers on $\Mbar_{g,n}$
by Liu-Xu \cite{Liu-Xu}.
The formula in \S \ref{sec:I2} was first discovered and checked
using Faber's Maple program \cite{Fab}.
In an earlier version of this paper,
it was stated as a conjecture and was only proved in some low genera cases.
The author thanks Professor Jim Bryan for encouraging him to find a proof.
The rest of the proof of Theorem \ref{thm:Main} is of combinatorial nature:
We take a seven-fold summation to arrive at our final answer.

An interesting byproduct is a relationship between Hurwitz-Hodge integrals with
polylogarithm function.
In the GRR relations of Mumford and Tseng,
Bernoulli numbers and Bernoulli polynomials evaluated at rational numbers appear respectively.
It is well-known that they are the values at negative integers of Riemann and Hurwitz
zeta functions respectively.
In the Hurwitz-Hodge integrals case
it turns out that we can rewrite the results in terms of polylogarithm functions.

For  type $D$ and $E$ binary polyhedral groups,
the simplifying trick used in this paper does not apply:
We have to compute integrals against Chern classes, not just one Chern character.
Nevertheless,
the method described in \cite{Zho} for computing Hurwitz-Hodge integrals can still be applied,
but the combinatorics is much more complicated.
We hope to address this in a future work.

In this paper,
we also consider the CRC for some 3D orbifolds.
It is natural to consider the orbifolds of the form $[\bC^2/G] \times \bC$,
where $G \subset SU(2)$ is a finite subgroup.
One can use some natural circle actions on these orbifolds
to define and study their equivariant orbifold Gromov-Witten invariants by virtual localization
\cite{Gra-Pan}.
We study the case of $G=\bZ_n$ in this paper.
In \S \ref{sec:3D} we specify some circle actions on the orbifold
$[\bC^2/\bZ_n] \times \bC$ used to define the potential function $F^{[\bC^2/\bZ_n]\times \bC}$
of equivariant orbifold Gromov-Witten invariant of this orbifold.

\begin{Theorem} \label{thm:Main2}
Up to polynomial terms in $u_1, \dots, u_{n-1}$,
we have
\be
F^{[\bC^2/\bZ_n]\times \bC}(\lambda;u_1, \dots, u_{n-1})
= \sum_{d \geq 1} \frac{1}{4 \sin^2(d\lambda)} \sum_{1 \leq s \leq t \leq n-1} \biggl(\xi_n^{t-s+1}
e^{v_{s\to t}} \biggr)^d
\ee
after analytic continuation,
where
\ben
v_{s\to t} = \sum_{a=s}^t v_a, \;\;\;\;
v_j = \frac{i}{n} \sum_{k=1}^{n-1} \sqrt{2 - 2 \cos \frac{2 k \pi}{n}} \xi_n^{jk} u_k.
\een
\end{Theorem}

On the other hand,
one can compute the Gromov-Witten invariants of $\widehat{\bC^2/\bZ_n} \times \bC$
by localization using the method of \cite{Zho0},
or use the theory of topological vertex \cite{Aga-Kle-Mar-Vaf, Li-Liu-Liu-Zho}.
One can simplify the expreesions by
the combinatorial techniques  in \cite{Zho1}.
Our result is (Theorem \ref{thm:F3D}):
\be
F^{\widehat{\bC^2/\bZ_n}\times \bC}(\lambda;Q_1, \dots, Q_{n-1})
= \sum_{1 \leq a \leq b \leq n-1} \sum_{d=1}^{\infty} \frac{\prod_{k=a}^b Q_k^d}{d}
\frac{1}{4 \sin^2 (d\lambda/2)}.
\ee
Hence CRC in this case take the following form (Theorem \ref{thm:CRC3D})
\be
F^{[\bC^2/\bZ_n]\times \bC}(\lambda;u_1, \dots, u_{n-1})
= F^{\widehat{\bC^2/\bZ_n}\times \bC}(\lambda;Q_1, \dots, Q_{n-1})
\ee
after analytic continuation and change of variables
$$Q_j = \xi_n e^{v_j}.$$
We make the following

\begin{Conjecture}
For a finite subgroup $G \subset SU(2)$,
the Crepant Resolution Conjecture takes the following form:
$$F^{[\bC^2/G] \times \bC}(\lambda;\{u_{\cl{g}}\}_{\cl{g} \neq \cl{1}})
= F^{G-\Hilb(\bC^2) \times \bC }(\lambda; \{Q_R\}_{R\;\text{nontrivial}})$$
after analytic continuation,
where
\ben
&& Q_R =  \exp \biggl( \frac{2\pi i \dim R}{|G|}  + v_R\biggr), \\
&& v_R = \frac{1}{|G|} \sum_{g\in G} \sqrt{\chi_V(g) - \dim V} \chi_R(g) u_{\cl{g}}.
\een
\end{Conjecture}

In a forthcoming work,
we will study CRC for more general three-dimensional Calabi-Yau orbifolds.

\section{Equivariant Gromov-Witten Invariant of $[\bC^2/\bZ_n]$}

In this section we define  the equivariant Gromov-Witten invariants
of $[\bC^2/\bZ_n]$,
and reduce their computations to intersection numbers on the Delgine-Mumford moduli spaces.

\subsection{Definition of the equivariant Gromov-Witten invariants of $[\bC^2/\bZ_n]$}

Because $[\bC^2/\bZ_n]$ is noncompact,
we define the Gromov-Witten invariants of $[\bC^2/\bZ_n]$ by taking suitable
torus action on $[\bC^2/\bZ_n]$ and using the virtual localization \cite{Gra-Pan}.

First of all, $\bZ_n$ acts on $\bC^2$ by:
$$\omega \cdot (z_1, z_2) = (\xi_n \cdot z_1, \xi_n^{-1} \cdot z_2),$$
where $\omega$ is a generator of $\bZ_n$,
and $\xi_n = e^{2\pi i/n}$.
This action has the origin as the only fixed point,
with normal bundle $V_1 \oplus V_{-1}$,
where $V_{\pm 1}$ are the one-dimensional representations
on which the generator $\omega$ of $\bZ_n$ acts by multiplication by $\xi_n^{\pm 1}$.

Let $\bC^*$ act on $\bC^2$ by multiplications.
The fixed locus of the induced action on the orbifold $[\bC^2/\bZ_n]$
is a copy of $\cB \bZ_n$,
the classifying stack of $\bZ_n$.
For $m \geq 1$ and $a_1, \dots, a_m \in \{0,1, \dots, n-1\}$ such that
\be \label{eqn:Monodromy}
\sum_{i=1}^m a_i \equiv 0 \pmod{n},
\ee
denote by $\Mbar_{g, m}([\bC^2/\bZ_n]; \coprod_{i=1}^m \cl{\omega^{a_i}}  )$
the moduli space of twisted stable maps to the orbifold
$[\bC^2/\bZ_n]$,
with monodromy  $\cl{\omega^{a_1}}, \dots, \cl{\omega^{a_m}}$ at the $m$ marked points.
Here  $\cl{\omega^k}$ denotes the conjugacy class of $\omega^k$.
See \cite{Abr-Gra-Vis} for definitions and notations.
The $\bC^*$-action on $[\bC^2/\bZ_n]$ induces a natural $\bC^*$-action on
the moduli space $\Mbar_{g, m}([\bC^2/\bZ_n];  \coprod_{i=1}^m \cl{\omega^{a_i}} )$.
Its fixed point set can be identified with
the moduli space $\Mbar_{g, m}(\cB \bZ_n;  \coprod_{i=1}^m \cl{\omega^{a_i}} )$
of twisted stable maps to $\cB \bZ_n$,
with monodromy  $\cl{\omega^{a_1}}, \dots, \cl{\omega^{a_m}}$ at the $m$ marked points.
Denote by $\bF^i_{\pm 1}$ the vector bundle on the moduli space associated
with the one-dimensional representations $V_{\pm 1}$.
The fibers of $\bF^i_{\pm 1}$ at a twisted stable map $f: C \to \cB\bZ_n$
is $(H^i(\tilde{C}, \cO_{\tilde{C}}) \otimes V_{\pm 1})^{\bZ_n}$,
where $\tilde{C} \to C$ is the admissible cover parameterized by $f$.

Denote by $p$ the number of $a_i$'s which are equal to zero.
It is not hard to see that when $a_i > 0$ for some $i$, i.e. $p < m$,
one has
$\bF^0_{\pm 1} = 0$.
Therefore,
by the dimension formula in \cite[Proposition 4.3]{Zho},
when $p < m$,
\bea
&& r_1:=\rank (\bF_1^1)  = g-1 + \frac{\sum_{i=1}^m a_i}{n}, \\
&& \bar{r}_1: = \rank (\bF_{-1}^1) = g-1 + \frac{\sum_{i=1}^m (n-a_i) (1- \delta_{a_i, 0})}{n},
\eea
and so
\be
r_1 + \bar{r}_1 = 2g-2 + m - p.
\ee
In this case,
we actually have $m - p \geq 2$, and so $r_1 + \bar{r}_1 \geq 0$.
If $r_1 + \bar{r}_1 = 0$,
then one has $g=0$ and $m-p = 2$.
However,
if all $a_i = 0$,
i.e., $p =m$,
then the situation is complicated.
Note $\Mbar_{g, m}(\cB \bZ_n; \cl{1}^m )$
has two components: $\Mbar_{g, m}^{disc}(\cB \bZ_n; \cl{1}^m )$
and $\Mbar_{g, m}^{conn}(\cB \bZ_n; \cl{1}^m )$.
A point in $\Mbar_{g, m}^{disc}(\cB \bZ_n; \cl{1}^m )$
is represented by a disconnected \'etale $\bZ_n$-cover $\tilde{C}$ of a stable curve $C$
in $\Mbar_{g,m}$,
consisting of n disjoint copies of $C$.
It follows that
$H^i(\tilde{C}, \cO_{\tilde{C}}) \otimes V_{\pm 1}$
is isomorphic to $H^i(C, \cO_C)$ (as a trivial $\bZ_n$ representation)
tensored with $V_{\pm 1}$ and the regular representation of $\bZ_n$,
hence
$$(H^i(\tilde{C}, \cO_{\tilde{C}}) \otimes V_{\pm 1})^{\bZ_n}
\cong H^i(C, \cO_C).$$
This means $\bF_{\pm 1}^1$ are just the pullback of the Hodge bundle
by the structure forgetting morphism:
$$\Mbar^{disc}_{g, m}(\cB \bZ_n; \cl{1}^m )
\to \Mbar_{g,m},$$
and both $\bF^0_1$ and $\bF^0_{-1}$ are isomorphic to the trivial line bundle.
On the other hand,
a point in $\Mbar_{g, m}^{conn}(\cB \bZ_n; \cl{1}^m )$
is represented by a connected \'etale $\bZ_n$-cover $\tilde{C}$ of a stable curve $C$ in $\Mbar_{g,m}$.
It follows that
$H^0(\tilde{C}, \cO_{\tilde{C}} \otimes V_{\pm 1})$
is isomorphic to $H^0(C, \cO_C)$ (as a trivial $\bZ_n$ representation)
tensored with $V_{\pm 1}$,
hence
$$(H^0(\tilde{C}, \cO_{\tilde{C}}) \otimes V_{\pm 1})^{\bZ_n}
\cong V_{\pm 1}^{\bZ_n} = 0.$$
This means
 both $\bF^1_1$ and $\bF^1_{-1}$ have rank $g-1$ by the orbifold Riemann-Roch formula.

Now $\bF_1^0+\bF^0_{-1}-\bF_1^1 - \bF^1_{-1}$ is the virtual normal bundle of
$\Mbar_{g, m}(\cB \bZ_n;  \coprod_{i=1}^m \cl{\omega^{a_i}} )$
in $\Mbar_{g, m}([\bC^2/ \bZ_n];  \coprod_{i=1}^m \cl{\omega^{a_i}} )$.
The torus $\bC^*$ acts on vector bundles $\bF_{\pm 1}^i$,
and their equivariant top Chern classes are given by their Chern polynomials.
Therefore,
using virtual localizatons \cite{Gra-Pan},
the equivariant correlators are given by:
\ben
&& \cor{ \prod_{j=1}^{m} \tau_{k_j}(e_{\cl{\omega^{a_j}}} )}_g^{[\bC^2/\bZ_n]} \\
& = & \int_{\Mbar_{g, m}(\cB \bZ_n;  \coprod_{i=1}^m \cl{\omega^{a_i}} )}
\frac{c_t(\bF_1^1) c_t(\bF_{-1}^1)}{c_t(\bF_1^0) c_t(\bF_{-1}^0)}
\cdot \prod_{j=1}^m \bpsi_j^{k_j},
\een
where $k_1, \dots, k_m \geq 0$ such that $\sum_{i=1}^m k_i = g + p$.
We will only consider the case of $m-p > 0$.
In this case we have $m - p \geq 2$ and $\bF_1^0 = \bF_{-1}^0 = 0$.

\subsection{Manipulations with $c_{t_1}(\bF^1_1) c_{t_2}(\bF^1_{-1})$}

\begin{lemma}
The product of the Chern polynomials of $\bF_1^1$ and $\bF_{-1}^1$ has the following expansion:
\be \label{eqn:Chern}
\begin{split}
& c_{t_1}(\bF_1^1) c_{t_2}(\bF_{-1}^1) \\
= & (t_1+t_2) (-1)^{r_1-1} (r_1+\bar{r}_1 -1)! \ch_{r_1 + \bar{r}_1 - 1}(\bF_1^1) + \cdots.
\end{split}
\ee
\end{lemma}

\begin{proof}
Recall the Hurwitz-Hodge bundles satisfy an analogue of Mumford's relations \cite{Mum, Bry-Gra-Pan}:
\be \label{eqn:BGP-Mumford}
c_t(\bF_1^1)c_{-t}(\bF_{-1}^1) = (-1)^{\bar{r}_1} t^{r_1+\bar{r}_1}.
\ee
In particular,
\bea
&& c_{r_1}(\bF_1^1) c_{\bar{r}_1}(\bF_{-1}^1) = 0, \\
&&  c_{r_1-1}(\bF_1^1) c_{\bar{r}_1}(\bF_{-1}^1) =  c_{r_1}(\bF_1^1) c_{\bar{r}_1-1}(\bF_{-1}^1). \label{eqn:CC}
\eea

Now we recall the relationship between Newton polynomials and elementary symmetric polynomials
(see e.g. \cite{Mac}).
Denote by $e_i(u_1, \dots, u_n)$ the $i$-th elementary symmetric function in $u_1, \dots, u_n$
and by
$$p_k(u_1, \dots, u_n) = u_1^k + \cdots + u_n^k$$
the $k$-th Newton symmetric polynomial in $u_1, \dots, u_n$.
Then one has
\ben
&& \log \sum_{i=0}^n t^i e_i(x_1, \dots, x_n)
= \sum_{i=1}^n  \log(1 + tx_i)
= \sum_{k=1}^{\infty} \frac{(-1)^{k-1}t^k}{k} p_k(x_1, \dots, x_n).
\een
Take derivative in $t$ on both sides:
\ben
\sum_{k=1}^{\infty} (-t)^{k-1} p_k(x_1, \dots, x_n)
= \frac{\sum_{i=1}^n i t^{i-1} e_i(x_1, \dots, x_n)}{\sum_{i=0}^n t^i e_i(x_1, \dots, x_n)}.
\een
Applying this to $c_i(\bF_1^1)$ and $k!\ch_k(\bF_1^1)$ we get:
\ben
\sum_{k=1}^{\infty} (-t)^{k-1} k! \ch_k(\bF_1^1)
= \frac{\sum_{i=1}^{r_1} i t^{i-1} c_i(\bF_1^1)}{\sum_{j=0}^{r_1} t^j c_j(\bF_1^1)}
= \sum_{i=1}^{r_1} i t^{i-1} c_i(\bF_1^1) \cdot \sum_{j=0}^{\bar{r}_1} (-t)^j c_j(\bF_{-1}^1),
\een
where we have used (\ref{eqn:BGP-Mumford}) in the last equality,
therefore we have
\be
k! \ch_k(\bF_1^1) = \sum_{i+j=k} (-1)^{i-1} i c_i(\bF_1^1)c_{j}(\bF_{-1}^1).
\ee
In particular, $\ch_k(\bF_1^1) = 0$
for $k \geq r_1 + \bar{r}_1$, and combining with (\ref{eqn:CC}) we have
\be \label{eqn:ChCl}
(r_1+\bar{r}_1-1)! \ch_{r_1+\bar{r}_1-1}(\bF_1^1)
= (-1)^{r_1-1} c_{r_1-1}(\bF_1^1)c_{\bar{r}_1}(\bF_{-1}^1).
\ee
Therefore,
\ben
&& c_{t_1}(\bF_1^1) c_{t_2}(\bF_{-1}^1) \\
& = & (c_{r_1}(\bF_1^1) + t_1 c_{r_1-1}(\bF_1^1) + \cdots + t_1^{r_1})
\cdot ( c_{\bar{r}_1}(\bF_{-1}^1) + t_2 c_{\bar{r}_1-1}(\bF_{-1}^1) + \cdots  + t_2^{\bar{r}_1}  ) \\
& = & (t_1+t_2) c_{r_1-1}(\bF_1^1) c_{\bar{r}_1}(\bF_{-1}^1) + \cdots \\
& = & (t_1+t_2) (-1)^{r_1-1} (r_1+\bar{r}_1 -1)! \ch_{r_1 + \bar{r}_1 - 1}(\bF_1^1) + \cdots.
\een
\end{proof}

\subsection{Computations of  Hurwitz-Hodge integrals}
By (\ref{eqn:Chern}) we have for $\sum_{i=1}^m k_i = g+p$,
\ben
&& \cor{ \prod_{j=1}^{m} \tau_{k_j}(e_{\cl{\omega^{a_j}}} )}_g^{[\bC^2/\bZ_n]} \\
& = & (-1)^{r_1-1} 2t(r_1+\bar{r}_1-1)! \int_{\Mbar_{g, m}(\cB \bZ_n;  \coprod_{i=1}^m \cl{\omega^{a_i}} )}
\ch_{r_1+\bar{r}_1-1}(\bF_1^1)  \cdot
\prod_{j=1}^m \bpsi_j^{k_j}.
\een
The Hurwitz-Hodge integral on the right-hand side can be computed by the method described in \cite{Zho}.
Write $[m]=\{1, \dots, m\}$.
By Tseng's GRR relations for Hurwitz-Hodge integrals \cite{Tse}:
\ben
&& \int_{\Mbar_{g, m}(\cB \bZ_n;  \coprod_{i \in [m]} \cl{\omega^{a_i}} )}
\ch_{r_1+\bar{r}_1-1}(\bF_1^1)  \cdot
\prod_{j=1}^m \bpsi_j^{k_j} \\
& = & - \frac{B_{r_1+\bar{r}_1}}{(r_1+\bar{r}_1)!}
\int_{\Mbar_{g, m+1}(\cB \bZ_n;  \coprod_{i \in [m]} \cl{\omega^{a_i}}, \cl1)}
\prod_{j=1}^{m} \bpsi_j^{k_j}
\cdot \bpsi_{m+1}^{r_1+\bar{r}_1} \\
& + & \sum_{j=1}^{m} \frac{B_{r_1+\bar{r}_1}(a_j/n)}{(r_1+\bar{r}_1)!}
\int_{\Mbar_{g, m}(\cB \bZ_n;  \coprod_{i \in [m]} \cl{\omega^{a_i}} )}
\prod_{i=1}^{m}  \bpsi_i^{k_i} \cdot
 \bpsi_j^{r_1+\bar{r}_1-1} \\
& - & \half \sum_{g_1+g_2=g} \sum_{I \coprod J =[m]}
 \frac{B_{r_1+\bar{r}_1}(c(a_I)/n)}{(r_1+\bar{r}_1)!} \\
&& \cdot \sum_{l=0}^{r_1+\bar{r}_1-2} (-1)^l
\int_{\Mbar_{g_1, |I| +1}(\cB\bZ_n; \coprod_{i \in I} \cl{\omega^{a_i}}, \cl{\omega^{c(I)}})}
\prod_{i\in I} \bpsi_i^{k_i} \cdot \bpsi_{|I|+1}^{r_1+\bar{r}_1-2-l} \\
&& \cdot n \cdot
\int_{\Mbar_{g_2, |J|+ 1}(\cB\bZ_n;  \coprod_{i \in J} \cl{\omega^{a_i}},  \cl{\omega^{-c(I)}})}
\prod_{j \in J} \bpsi_j^{k_j} \cdot \bpsi_{|J|+1}^{l} \\
& - & \half \sum_{c=0}^{n-1}
 \frac{B_{r_1+\bar{r}_1}(c/n)}{(r_1+\bar{r}_1)!}
\cdot n \cdot \sum_{l=0}^{r_1+\bar{r}_1-2} (-1)^l \\
&& \cdot \int_{\Mbar_{g-1, m +2}(\cB\bZ_n; \coprod_{i \in [m]} \cl{\omega^{a_i}}, \cl{\omega^c}, \cl{\omega^{-c}})}
\prod_{i\in [m]} \bpsi_i^{k_i} \cdot \bpsi_{m+1}^{r_1+\bar{r}_1-2-l} \cdot \bpsi_{m+2}^{l},
\een
where $c(a_I) =0, 1, \dots, n-1$,
such that $c(a_I) \equiv  - \sum_i a_i  \pmod{n}$.
Now by Jarvis-Kimura \cite{Jar-Kim},
\ben
&& \int_{\Mbar_{g, m}(\cB \bZ_n; \coprod_{i\in [m]} \cl{\omega^{a_i}} )}
\ch_{r_1+\bar{r}_1-1}(\bF_1^1)  \cdot
\prod_{j=1}^{n} \bpsi_j^{k_j} \\
& = & n^{2g-1} \biggl(
- \frac{B_{r_1+\bar{r}_1}}{(r_1+\bar{r}_1)!}
\cor{ \prod_{j=1}^{m} \tau_{k_j} \cdot \tau_{r_1+\bar{r}_1}}_g \\
& + & \sum_{j=1}^{m}  \frac{B_{r_1+\bar{r}_1}(a_j/n)}{(r_1+\bar{r}_1)!}
\cor{\prod_{i=1}^{m}  \tau_{k_i + \delta_{ij} (r_1 + \bar{r}_1 -1)} }_g \\
& - & \half \sum_{\substack{g_1+g_2=g, I \coprod J =[m]\\2g_1+2+|I| >0, 2g_2-2+|J|>0}}
\frac{B_{r_1+\bar{r}_1}(c(a_I)/n)}{(r_1+\bar{r}_1)!} \\
&& \cdot \sum_{l=0}^{r_1+\bar{r}_1-2} (-1)^l\cor{\prod_{i\in I} \tau_{k_i} \cdot \tau_{r_1+\bar{r}_1-2-l}}_{g_1}
\cdot \cor{ \tau_l \cdot \prod_{j \in J} \tau_{k_j} }_{g_2} \\
& - & \half \sum_{c=0}^{n-1}  \frac{B_{r_1+\bar{r}_1}(c/n)}{(r_1+\bar{r}_1)!}
 \sum_{l=0}^{r_1+\bar{r}_1-2} (-1)^l\cor{\prod_{i\in [m]} \tau_{k_i} \cdot \tau_{r_1+\bar{r}_1-2-l}
 \cdot \tau_l}_{g-1} \biggr).
\een
Here as usual,
$$\langle \tau_{a_1} \cdots \tau_{a_k} \rangle_g
= \int_{\Mbar_{g, k}} \psi_1^{a_1} \cdots \psi_k^{a_k},
$$
where $2g-2+k > 0$.
Because $\langle \tau_{a_1} \cdots \tau_{a_k} \rangle_g = 0$ except for
$a_1 + \cdots + a_k = 3g-3+ k$,
the summation over $l$ in the third term on the right-hand side
can be reduced to the case of $l= 3g_2-2 - \sum_{j \in J} (k_j-1)$.
Hence,
\ben
&& \int_{\Mbar_{g, m}(\cB \bZ_n; \coprod_{i\in [m]} \cl{\omega^{a_i}} )}
\ch_{r_1+\bar{r}_1-1}(\bF_1^1)  \cdot
\prod_{j=1}^{n} \bpsi_j^{k_j} \\
& = & n^{2g-1} \biggl(
- \frac{B_{r_1+\bar{r}_1}}{(r_1+\bar{r}_1)!}
\cor{ \prod_{j=1}^{m} \tau_{k_j} \cdot \tau_{r_1+\bar{r}_1}}_g \\
& + & \sum_{j=1}^{m}  \frac{B_{r_1+\bar{r}_1}(a_j/n)}{(r_1+\bar{r}_1)!}
\cor{\prod_{i=1}^{m}  \tau_{k_i + \delta_{ij} (r_1 + \bar{r}_1 -1)} }_g \\
& - & \half \sum_{\substack{g_1+g_2=g, I \coprod J =[m]\\2g_1+2+|I| >0, 2g_2-2+|J|>0}}
\frac{B_{r_1+\bar{r}_1}(c(a_I)/n)}{(r_1+\bar{r}_1)!}   (-1)^{g_2 + \sum_{j\in J} (k_j-1)} \\
&& \cdot
\cor{\prod_{i\in I} \tau_{k_i} \cdot \tau_{3g_1-2-\sum_{i\in I} (k_i-1)}}_{g_1}
\cdot \cor{ \tau_{3g_2-2-\sum_{j\in J} (k_j-1)} \cdot \prod_{j \in J} \tau_{k_j} }_{g_2} \\
& - & \half \sum_{c=0}^{n-1}  \frac{B_{r_1+\bar{r}_1}(c/n)}{(r_1+\bar{r}_1)!}
 \sum_{l=0}^{r_1+\bar{r}_1-2} (-1)^l\cor{\prod_{i\in [m]} \tau_{k_i} \cdot \tau_{r_1+\bar{r}_1-2-l}
 \cdot \tau_l}_{g-1} \biggr).
\een
Even though the first three terms on the right-hand side look differently,
they can be written in a uniform way if we use the following conventions:
\be \label{eqn:Conv1}
\cor{\tau_l}_0 = \begin{cases}
1, & l = -2, \\
0, & \text{otherwise},
\end{cases}
\ee
and
\be \label{eqn:Conv2}
\cor{\tau_k \tau_l}_0
= \begin{cases}
(-1)^k, & l = -k -1, k \geq 0, \\
(-1)^l, & k = - l - 1, l \geq 0, \\
0, & \text{otherwise}.
\end{cases}
\ee
To simplify the notations further,
we introduce for $I \coprod J = [m]$:
\ben
&& \{ \prod_{i\in I} \tau_{k_i} | \prod_{j \in J} \tau_{k_j} \}_g \\
& = &  (-1)^{\sum_{j \in J} k_j} \sum_{g_1+g_2=g}
 (-1)^{g_2} \cor{\prod_{i\in I} \tau_{k_i} \cdot \tau_{3g_1-2+|I|-\sum_{i\in I} k_i}}_{g_1} \\
&& \cdot \cor{\prod_{j \in J} \tau_{k_j} \cdot \tau_{3g_2-2+|J|- \sum_{j \in J} k_j}}_{g_2},
\een
and
\ben
&& [\prod_{i\in [m]} \tau_{k_i}]^K_{g-1}
=  \sum_{l=0}^K (-1)^l \cor{\prod_{i\in [m]} \tau_{k_i} \cdot \tau_{K-l}
 \cdot \tau_l}_{g-1}.
\een
Recall $r_1 + \bar{r}_1 = 2g-2 + m - p$.
Then we have the following:

\begin{proposition}
For $g \geq 0$:
\ben
&& \int_{\Mbar_{g, m}(\cB \bZ_n; \coprod_{i \in [m]} \cl{\omega^{a_i}} )}
\ch_{r_1+\bar{r}_1-1}(\bF_1)  \cdot
\prod_{j=1}^{n} \bpsi_j^{k_j} \\
& = & - \half n^{2g-1} \biggl(
\sum_{I \coprod J =[m]}   \frac{B_{2g-2+m-p}(c(a_I)/n)}{(2g-2+m-p)!}
\cdot (-1)^{|J|} \{\prod_{i\in I} \tau_{k_i}|\prod_{j \in J} \tau_{k_j}\}_g \\
& + & \sum_{c=0}^{n-1}  \frac{B_{2g-2+m-p}(c/n)}{(2g-2+m-p)!}
[\prod_{i\in [m]} \tau_{k_i}]_{g-1}^{2g-4+m-p} \biggr).
\een
\end{proposition}

\subsection{Results on $[\prod_{i\in [m]} \tau_{k_i}]^{2g-4+m-p}_{g-1}$}
\label{sec:I1}

Recent work of Liu-Xu \cite{Liu-Xu} contains explicit formula for
$[\prod_{i\in [m]} \tau_{k_i}]^{2g-4+m-p}_{g-1}$.
We are concerned with the case of $m-p  \geq 2$.
For $m-p=2$,
first assume $k_i > 0$ for $i=1, \dots, m$,
then we are in the situation of \cite[Theorem 2.1]{Liu-Xu}:
\ben
[\prod_{i\in [m]} \tau_{k_i}]_{g-1}^{2g-2}
& = & \sum_{l=0}^{2g-2} (-1)^l \cor{\prod_{i\in [m]} \tau_{k_i} \cdot \tau_{2g-2-l}\tau_l}_{g-1} \\
& = & \frac{(2g-3+m)!}{4^{g-1}(2g-1)!} \cdot \frac{1}{\prod_{i=1}^m (2k_i-1)!!};
\een
otherwise, by string equation and induction we get:
\ben
[\tau_0^a \prod_{i=1}^{m-a} \tau_{k_i}]_{g-1}^{2g-2}
& = & \sum_{l=0}^{2g-2} (-1)^l \cor{\tau_0^k
\prod_{i=1}^{m-a} \tau_{k_i} \cdot \tau_{2g-2-l}\tau_l}_{g-1} \\
& = & \frac{(2g-3+m-a)!}{4^{g-1}(2g-1)!} \cdot
\frac{\prod_{j=1}^a (2g-4 + m -a + 2j)}{\prod_{i=1}^{m-a} (2k_i-1)!!},
\een
where $k_i > 0$ for $i=1, \dots, m-a$.
For $m - p > 2$,
we are in the situation of \cite[Theorem 2.3]{Liu-Xu}:
$$ [\prod_{i\in [m]} \tau_{k_i}]_{g-1}^{2g-4+m-p} = 0,$$
where $k_i \geq 0$ for $i=1, \dots, m$.

\subsection{Results on $\{\prod_{i\in I} \tau_{k_i} | \prod_{j \in J} \tau_{k_j}\}_g $}
\label{sec:I2}

For $p= 0$, we have

\begin{proposition}
For $g \geq 0$ and $k_1, \dots, k_m \geq 0$ such that $k_1  + \cdots + k_m = g$,
and $I \coprod J = [m]$,
the following identities holds:
\be \label{eqn:I2}
\{\prod_{i\in I} \tau_{k_i} | \prod_{j \in J} \tau_{k_j}\}_g
=  \frac{1}{4^g \prod_{j=1}^m (2k_j+1)!!}.
\ee
\end{proposition}

\begin{proof}
This can be again proved by the recent results of Liu-Xu \cite{Liu-Xu}.
Consider the normalized $n$-point function:
$$G(x_1, \dots, x_n; \lambda) = \exp \biggl(-\frac{\sum_{j=1}^n x_j^3\lambda^2}{24} \biggr) \cdot
F(x_1, \dots, x_n; \lambda),$$
where
\ben
F(x_1,\dots,x_n; \lambda)
& = & \sum_{g=0}^{\infty} \lambda^{2g}
\sum_{\substack{d_1, \dots, d_n \geq 0 \\ \sum d_j=3g-3+n}}
\cor{\tau_{d_1}\cdots\tau_{d_n} }_g \prod_{j=1}^n x_j^{d_j} \\
& + & \delta_{n,1} \cor{\tau_{-2}}_0 x_1^{-2}
+ \delta_{n, 2} \sum_{l=0}^{\infty} \cor{\tau_{-l-1}\tau_{l}}_0 x_1^{-l-1} x_2^l
\een
is the $n$-point function.
Denote by $\cC(\prod_{i=1}^n x_i^{d_i}, P(x_1, \dots, x_n))$
the coefficient of $\prod_{i=1}^n x_i^{d_i}$ in a polynomial or a formal power series
$P(x_1, \dots, x_n)$.
It is easy to see from the definitions that for $d_1, \dots, d_m \geq 0, d_1 + \cdots + d_m = g$,
we have
\ben
&& \cC(\lambda^{2g} z^{2g-4+m} \prod_{i=1}^m x_i^{d_i}, G(z, x_I; \lambda) G(-z, x_J; \lambda)) \\
& = & \sum_{g_1+g_2=g} (-1)^{3g_2 - 2 + \sum_{i \in J} (d_i - 1)}
\cor{\prod_{i\in I} \tau_{d_i} \cdot \tau_{3g_1-2- \sum_{i\in I} (d_i -1)}}_{g_1} \\
&& \cdot
\cor{\prod_{i\in J} \tau_{d_i} \cdot \tau_{3g_2-2- \sum_{i\in J} (d_i -1)}}_{g_2} \\
& = & (-1)^{|J|} \{ \prod_{i\in I} \tau_{k_i} | \prod_{j \in J} \tau_{k_j} \}_g.
\een
On the other hand,
by \cite[Theorem 3.2]{Liu-Xu},
for $d_1, \dots, d_n \geq 0$, $\sum_{i=1}^n d_i = 3g-2+n-k$ one has
\ben
\cC(\lambda^{2g}z^k\prod_{j=1}^n x_j^{d_j}, G(z, x_1, \dots, x_n; \lambda))
= \begin{cases}
0, & k > 2g-2+n, \\
\frac{1}{4^g \cdot \prod_{j=1}^n (2d_j + 1)!!}, & k = 2g-2 + n.
\end{cases}
\een
The proof is completed by applying this to $G(z,x_I;\lambda) \cdot G(-z, x_J;\lambda)$.
\end{proof}

Unfortunately we do not have a closed formula for the $p > 0$ case.

\subsection{Potential function}
\label{sec:Potential}

Define the genus $g$ equivariant orbifold Gromov-Witten potential function by:
\ben
&& F_g^{[\bC^2/\bZ_n]}(\{x_{i,k}\}_{0 \leq i \leq n-1, k \geq 0}; u) \\
& = & \sum_{m \geq 1} \frac{1}{m!} \sum_{\substack{0 \leq a_1, \dots, a_m \leq n-1 \\ \sum_j a_j \equiv 0 \pmod{n}}}
\sum_{\sum_j k_j = g + p}
\cor{ \prod_{j=1}^{m} \tau_{k_j}(e_{\cl{\omega^{a_j}}} )}_g^{[\bC^2/\bZ_n]}
 \cdot \prod_{j=1}^m x_{a_j, k_j} \cdot \prod_{a_i > 0} u_{a_i},
\een
where $x_{i, k}$ ($0 \leq i \leq n-1$, $k \geq 0$) and $u_j$ ($1 \leq j \leq n-1$) are formal variables.
We understand $\{x_{i,k}\}$ as linear coordinates on $H^*_{orb}([\bC^2/\bZ_n])$.
The variables $u_1, \dots, u_{n-1}$ are referred to as the degree tracking variables.
We understand that $u_0 = 1$.
Then we have by \S 2.3:
\ben
&& F_g^{[\bC^2/\bZ_n]}(\{x_{i,k}\}_{0 \leq i \leq n-1, k \geq 0}; u) \\
& = & t  n^{2g-1}  \sum_{m \geq 1} \frac{1}{m!}
\sum_{\sum_{j=1}^m a_j \equiv 0 \pmod{n}} (-1)^{r_1} \sum_{\sum_j k_j = g+p} \biggl(
 \sum_{I \coprod J =[m]}   \frac{B_{r_1+\bar{r}_1}(c(I)/n)}{r_1+\bar{r}_1} \\
&& \cdot (-1)^{|J|}
\{\prod_{i\in I} \tau_{k_i} | \prod_{j \in J} \tau_{k_j}\}_g
+ \sum_{c=0}^{n-1}  \frac{B_{r_1+\bar{r}_1}(c/n)}{r_1+\bar{r}_1}
[\prod_{i\in [m]} \tau_{k_i}]^{r_1+\bar{r}_1-2}_{g-1} \biggr) \\
&& \cdot \prod_{j=1}^m x_{a_j, k_j} \cdot \prod_{i = 1}^m u_{a_i},
\een
modulo terms of the form $\prod_{j=1}^m x_{0, k_j}$.

Now we apply the results on intersection numbers in \S \ref{sec:I1} and \S \ref{sec:I2}
to get the following formula for the stationary part of the potential function
(setting $x_{0,k} = 0$ for $k \geq 0$):
\ben
&& F_g^{[\bC^2/\bZ_n]}(\{x_{i,k}\}_{1 \leq i \leq n-1, k \geq 0}; u) \\
& = & t  n^{2g-1} \biggl(  \sum_{m \geq 1} \frac{1}{m!}
\sum_{\substack{1 \leq a_1, \dots, a_m \leq n-1 \\\sum_{j=1}^m a_j \equiv 0 \pmod{n}}} (-1)^{r_1} \sum_{\sum_j k_j = g}
 \sum_{I \coprod J =[m]}   \frac{B_{r_1+\bar{r}_1}(c(I)/n)}{r_1+\bar{r}_1} \\
&& \cdot (-1)^{|J|}
\cdot \frac{1}{4^g \prod_{j=1}^m (2k_j+1)!!} \cdot \prod_{j=1}^m x_{a_j, k_j}
 \cdot \prod_{i=1}^m u_{a_i} \\
& + & \frac{1}{2}
\sum_{a=1}^{n-1} (-1)^g \sum_{\sum_j k_j = g} \sum_{c=0}^{n-1}
\frac{B_{2g}(c/n)}{2g} \cdot \frac{1}{4^{g-1}}
\frac{1}{\prod_{k_i > 0} (2k_i-1)!!}
\cdot x_{a, k_1} x_{n-a, k_2} u_a u_{n-a} \biggr).
\een

\section{Crepant Resolution Conjecture for Type A Surface Singularities in All Genera}

In this section we perform the change of variables and analytic continuations
to the stationary potential functions $F_g^{[\bC^2/\bZ_n]}$.
We will compare with results of Maulik \cite{Mau} to establish
the CRC in all genera for the stationary potential functions.

\subsection{Modified stationary potential function}

To simplify the four-fold summation in the expression for
the stationary potential function
$F_g^{[\bC^2/\bZ_2]}(\{x_{i,k}\}_{1 \leq i \leq n-1, k \geq 0};u)$,
we will first convert it to a seven-fold summation.

We first group the terms with $c(I) = c$ to get a constrained summation
$$\sum_{c=0}^{n-1}\sum_{c(I)= c},$$
we then replace it by an equivalent arbitrary summation
$\sum_{c=0}^{n-1}\frac{1}{n} \sum_{l=0}^{n-1} \xi_n^{lc} \xi_n^{l\sum_{i \in I} a_i}$.
We also replace the constrained summation
$$\sum_{\substack{1 \leq a_1, \dots, a_m \leq n-1 \\ \sum_j a_j \equiv 0 \pmod{n}}}$$
by an equivalent arbitrary summation
$$\sum_{1 \leq a_1, \dots, a_m \leq n-1} \frac{1}{n} \sum_{b=0}^{n-1} \xi_n^{b\sum_i a_i},$$
now we get
\ben
&& F_g^{[\bC^2/\bZ_n]}(\{x_{i, k}\}_{1 \leq i \leq n-1, k \geq 0}; u) \\
& = & t n^{2g-1} \biggl( \sum_{m \geq 1} \frac{1}{m!}
\sum_{1 \leq a_1, \dots, a_m \leq n-1}  \frac{1}{n}
\sum_{b=0}^{n-1} \xi_n^{b \sum_i a_i}
(-1)^{r_1}\sum_{\sum_j k_j = g} \sum_{I \coprod J =[m]}\\
&& \sum_{c=0}^{n-1}
\frac{1}{n} \sum_{l=0}^{n-1} \xi_n^{lc} \xi_n^{l\sum_{i \in I} a_i}
\cdot \frac{B_{2g-2+m}(c/n)}{2g-2+m}
\cdot (-1)^{|J|}\\
&& \cdot \frac{1}{4^g \prod_{j=1}^m (2k_j+1)!!} \cdot \prod_{j=1}^m (x_{a_j, k_j} u_{a_j})
+ \cdots \biggr).
\een
Finally the summation over $k_i$'s is constrained by the condition that
$\sum_i k_i =g$.
To convert to an arbitrary summation we introduce an extra genus tracking variable $z$,
and consider the modified stationary potential function:
\ben
&& \tilde{F}_g^{[\bC^2/\bZ_n]}( \{x_{i,k}\}_{1 \leq i \leq n-1, k \geq 0};u; z) \\
& = & t n^{2g-1}  \sum_{m} \frac{1}{m!}
\sum_{1 \leq a_1, \dots, a_m \leq n-1}  (-1)^{g-1+\sum_i a_i/n}\frac{1}{n}
\sum_{b=0}^{n-1} \xi_n^{b\sum_i a_i} \sum_{k_1, \dots, k_m \geq 0}\\
&& \cdot \sum_{I \coprod J =[m]} \sum_{c=0}^{n-1}
\frac{1}{n} \sum_{l=0}^{n-1} \xi_n^{lc} \xi_n^{l\sum_{i \in I} a_i}
 \frac{B_{2g-2+m}(c/n)}{2g-2+m} (-1)^{|J|}
 \prod_{j=1}^m \frac{x_{a_j, k_j}z^{k_j }u_{a_j}}{4^{k_j} (2k_j+1)!!},
\een
whose coefficient of $z^g$ is equal to
$ F_g^{[\bC^2/\bZ_n]}(\{x_{i,k}\}_{1\leq i \leq n-1, k \geq 0}; u)$,
up to some quadratic terms.

Now we take care of the seven summations in
$\tilde{F}_g^{[\bC^2/\bZ_n]}( \{x_{i,k}\}_{0 \leq i \leq n-1, k \geq 0}; u; z) $,
in a suitable order.
First we take the summation $\sum_{I \coprod J = [m]}$.
Because
\ben
&& \sum_{I \coprod J = [m]} \xi_n^{l \sum_{i \in I} a_i} (-1)^{|J|}
= \prod_{j=1}^m (\xi_n^{la_j} - 1),
\een
we have
\ben
&& \tilde{F}_g^{[\bC^2/\bZ_n]}( \{x_{i,k}\}_{1 \leq i \leq n-1, k \geq 0};u; z)  \\
& = & (-1)^{g-1} t n^{2g-3} \sum_{b=0}^{n-1}\sum_{l=0}^{n-1} \sum_{m} \frac{1}{m!}
\sum_{1 \leq a_1, \dots, a_m \leq n-1} \sum_{k_1, \dots, k_m \geq 0} \xi_{2n}^{\sum_i a_i}   \xi_n^{b\sum_i a_i} \\
&& \cdot \prod_{j=1}^m (\xi_n^{la_j} - 1)  \cdot \prod_{j=1}^m \frac{x_{a_j, k_j}u_{a_j}z^{k_j}}{4^{k_j} (2k_j+1)!!}
\sum_{c=0}^{n-1}  \xi_n^{lc} \cdot \frac{B_{2g-2+m}(c/n)}{2g-2+m}.
\een
Then we take the summations $\sum_{1 \leq a_1, \dots, a_m \leq n-1} \sum_{k_1, \dots, k_m \geq 0}$ to get:
\be \label{eqn:F}
\begin{split}
&  \tilde{F}_g^{[\bC^2/\bZ_n]}( \{x_{i,k}\}_{0 \leq i \leq n-1, k \geq 0};u; z)  \\
 = & (-1)^{g-1} t n^{2g-3} \sum_{b=0}^{n-1} \sum_{l=1}^{n-1}\sum_{m} \frac{1}{m!} \biggl(
\sum_{1 \leq a \leq n-1} \sum_{k \geq 0} \xi_{2n}^a   \xi_n^{ba}
 (\xi_n^{la} - 1)  \cdot \frac{x_{a, k}u_a z^{k}}{4^k(2k+1)!!} \biggr)^m \\
& \cdot \sum_{c=0}^{n-1}  \xi_n^{lc} \cdot \frac{B_{2g-2+m}(c/n)}{2g-2+m}.
\end{split}
\ee

\subsection{Reformulation in terms of the polylogarithm function}
In this subsection we take care of summations $\sum_{c=0}^{n-1}$.
The result  will not be used below but it may have some independent interest.
Recall Hurwitz zeta function is defined by:
\be
\zeta(s, a) = \sum_{n=0}^\infty \frac{1}{(n+a)^s}.
\ee
It is well-known that
\be \label{eqn:HurwitzZeta}
\zeta(-n, a) = - \frac{B_{n+1}(a)}{n+1}.
\ee
The polylogarithm function is defined by:
\be
\Li_s( x) = \sum_{n=1}^{\infty} \frac{x^n}{n^s}.
\ee

\begin{lemma}
For $0 \leq l \leq n-1$,
\be \label{eqn:Hur2Ler}
\sum_{a=0}^{n-1} \xi_n^{al} \zeta(s, a/n) = n^s \cdot \Li_s(\xi_n^l).
\ee
\end{lemma}

\begin{proof}

\ben
&& \sum_{a=0}^{n-1} \xi_n^{al} \zeta(s, a/n)
= \sum_{m=1}^{\infty} \frac{1}{m^s}
+ \sum_{a=1}^{n-1} \sum_{m=0}^{\infty} \frac{\xi_n^{al}}{(m+a/n)^s} \\
& = & n^s \sum_{m=1}^{\infty} \frac{1}{(nm)^s}
+ n^s \sum_{a=1}^{n-1} \sum_{m=0}^{\infty} \frac{\xi_n^{al}}{(nm+a)^s}
= n^s \sum_{m=1}^{\infty} \frac{\xi_n^{ml}}{m^s}
= n^s \cdot \Li_s(\xi_n^l).
\een
\end{proof}

Combining (\ref{eqn:F}), (\ref{eqn:HurwitzZeta}) and (\ref{eqn:Hur2Ler}) we get:

\begin{proposition}
For $n \geq 2$ the following identity holds:
\be \label{eqn:FLerch}
\begin{split}
&  \tilde{F}_g^{[\bC^2/\bZ_n]}( \{x_{i,k}\}_{1 \leq i \leq n-1, k \geq 0}; u; z)  \\
 = & (-1)^{g} t \sum_{b=0}^{n-1} \sum_{l=1}^{n-1}\sum_{m} \frac{1}{m!} \biggl(
\sum_{0 \leq a \leq n-1} \sum_{k \geq 0} \xi_{2n}^a   \xi_n^{ba}
 (\xi_n^{la} - 1)  \cdot \frac{x_{a, k}u_a z^{k}}{4^k(2k+1)!!} \biggr)^m \\
& \cdot  n^{-m} \cdot \Li_{-(2g-3+m)}(\xi_n^l).
\end{split}
\ee
\end{proposition}

\subsection{Change of variables}
When $G=\bZ_n$,
the change of variables given by Bryan-Graber \cite{Bry-Gra} is:
\ben
&& q_1 = \cdots = q_{n-1} = \xi_n, \\
&& y_j = \frac{i}{n} \sum_{k=1}^{n-1} \sqrt{2 - 2 \cos \frac{2 k \pi}{n}} \xi_n^{jk}x_k
= \frac{2i}{n} \sum_{k=1}^{n-1} \sin \frac{k\pi}{n} \cdot \xi_n^{jk} x_k.
\een
Here $q_1, \dots, q_n$ are degree tracking variables for the minimal resolution of
$\bC^2/\bZ_n$,
and they have set all the degree tracking variables on $[\bC^2/\bZ_n]$ side to be $1$.
In the above we have considered the degree tracking variables in our partition function,
so we will take instead:
$$ y_j = \frac{2i}{n} \sum_{k=1}^{n-1} \sin \frac{k\pi}{n} \cdot \xi_n^{jk} x_k u_k.$$

\begin{lemma}
For $a=1,\dots, n-1$ we have
\be \label{eqn:XY}
x_a u_a= - \frac{i}{2 \sin \frac{a \pi}{n}} \sum_{j=1}^{n-1} (\xi_n^{-aj} - 1) y_j.
\ee
\end{lemma}

\begin{proof}
First notice that
\ben
&& \sum_{j=1}^{n-1} \xi_n^{-jl} y_j
= \sum_{j=1}^{n-1} \xi_n^{-jl} \frac{2i}{n}
\sum_{k=1}^{n-1} \sin \frac{k \pi}{n} \cdot \xi_n^{jk}x_ku_k
= \frac{2i}{n} \sum_{k=1}^{n-1}\sin \frac{k \pi}{n} \cdot (-1 + n \delta_{kl}) x_ku_k.
\een
It is more transparent when written in matrix form:
\ben
&& \begin{pmatrix}
\xi_n^{-1} & \xi_n^{-2} & \cdots & \xi_n^{-(n-1)} \\
\xi_n^{-2} & \xi_n^{-2\cdot 2} & \cdots & \xi_n^{-2(n-1)} \\
\vdots & \vdots & & \vdots \\
\xi_n^{-(n-1)} & \xi_n^{-(n-1)2} &\cdots & \xi_n^{-(n-1)^2}
\end{pmatrix} \cdot
\begin{pmatrix}
y_1 \\ y_2 \\ \vdots \\ y_{n-1}
\end{pmatrix} \\
& = & \frac{2i}{n}\begin{pmatrix}
n-1 & -1 & \cdots & -1 \\
-1 & n-1 & \cdots & -1 \\
\vdots & \vdots & & \vdots \\
-1 & -1 & \cdots & n-1
\end{pmatrix} \cdot
\begin{pmatrix}
\sin \frac{\pi}{n}\cdot x_1u_1 \\ \sin \frac{2 \pi}{n}\cdot x_2u_2
\\ \vdots \\ \sin \frac{(n-1) \pi}{n}  \cdot x_{n-1}u_{n-1}
\end{pmatrix}.
\een
Notice that
\ben
\begin{pmatrix}
n-1 & -1 & \cdots & -1 \\
-1 & n-1 & \cdots & -1 \\
\vdots & \vdots & & \vdots \\
-1 & -1 & \cdots & n-1
\end{pmatrix}^{-1}
= \frac{1}{n} \begin{pmatrix}
2 & 1 & \cdots & 1 \\
1 & 2 & \cdots & 1 \\
\vdots & \vdots & & \vdots \\
1 & 1 & \cdots & 2
\end{pmatrix},
\een
and
\ben
&& \begin{pmatrix}
2 & 1 & \cdots & 1 \\
1 & 2 & \cdots & 1 \\
\vdots & \vdots & & \vdots \\
1 & 1 & \cdots & 2
\end{pmatrix} \cdot
\begin{pmatrix}
\xi_n^{-1} & \xi_n^{-2} & \cdots & \xi_n^{-(n-1)} \\
\xi_n^{-2} & \xi_n^{-2\cdot 2} & \cdots & \xi_n^{-2(n-1)} \\
\vdots & \vdots & & \vdots \\
\xi_n^{-(n-1)} & \xi_n^{-(n-1)2} &\cdots & \xi_n^{-(n-1)^2}
\end{pmatrix} \\
& = &
\begin{pmatrix}
\xi_n^{-1}-1 & \xi_n^{-2}-1 & \cdots & \xi_n^{-(n-1)}-1 \\
\xi_n^{-2}-1 & \xi_n^{-2\cdot 2}-1 & \cdots & \xi_n^{-2(n-1)}-1 \\
\vdots & \vdots & & \vdots \\
\xi_n^{-(n-1)}-1 & \xi_n^{-(n-1)2}-1 &\cdots & \xi_n^{-(n-1)^2}-1
\end{pmatrix}
\een
Now it is straightforward to get (\ref{eqn:XY}).
\end{proof}

\subsection{A summation formula}
We now take care of the summation $\sum_{a=0}^{n-1}$ in (\ref{eqn:F}).

\begin{lemma} \label{lm:Combo}
For $0 \leq b \leq n-1$, $ 1 \leq l \leq n-1$,
we have
\be
\sum_{k=1}^{n-1} \xi_n^{bk}\xi_{2n}^k (\xi_n^{kl}-1)x_ku_k
= \begin{cases}
n y_{b+1\to b+l}, & b + l < n, \\
- n y_{b+l-n+1\to b}, & b + l \geq n,
\end{cases}
\ee
where for $1 \leq s \leq t \leq n-1$,
we define
$$y_{s\to t} = y_s + \cdots + y_t.$$
\end{lemma}

\begin{proof}
By (\ref{eqn:XY}),
\ben
&& \sum_{k=1}^{n-1} \xi_n^{bk}\xi_{2n}^k (\xi_n^{kl}-1)x_ku_k
= - \sum_{k=1}^{n-1} \xi_n^{bk}\xi_{2n}^k (\xi_n^{kl}-1) \frac{i}{2 \sin \frac{k \pi}{n}}
\sum_{j=1}^{n-1} (\xi_n^{-jk} -1) y_j \\
& = & - \sum_{j=1}^{n-1} K(b, l, j) y_j,
\een
where
\ben
K(b, l, j) & = & \sum_{k=1}^{n-1} \xi_n^{bk}\xi_{2n}^k (\xi_n^{kl}-1) \frac{i}{2 \sin \frac{k \pi}{n}}
(\xi_n^{-jk} -1).
\een
For $a \in \bZ$,
let $r_n(a)$ be the integer such that $0 \leq r_n(a) < n$ and $r_n(a) \equiv a \pmod{n}$.
\ben
K(b, l, j)
& = & \sum_{k=1}^{n-1} \xi_n^{bk}\xi_{2n}^k (\xi_n^{kl}-1) \frac{-1}{\xi_{2n}^k - \xi_{2n}^{-k}}
(\xi_n^{-jk} -1) \\
& = & - \sum_{k=1}^{n-1} \xi_n^{bk}\xi_{n}^k (\xi_n^{kl}-1) \frac{1}{\xi_{n}^k - 1} (\xi_n^{-jk} -1) \\
& = & - \sum_{k=1}^{n-1} \xi_n^{bk}\xi_{n}^k \sum_{c=0}^{l-1} \xi_n^{ck}  (\xi_n^{-jk} -1) \\
& = & - \sum_{c=0}^{l-1} \sum_{k=0}^{n-1} \xi_n^{(b+c-j+1)k}
+ \sum_{c=0}^{l-1} \sum_{k=0}^{n-1} \xi_n^{(b+c+1)k} \\
& = & -n \sum_{c=0}^{l-1} \delta_{c, r_n(j-b-1)}
+ n \sum_{c=0}^{l-1} \delta_{c, n-b-1}.
\een
It is then easy to see that when $n-b-1 > l -1$, i.e., $b + l < n$,
\ben
K(b, l, j) = \begin{cases}
-n, & b+1 \leq j \leq l+ b, \\
0, & \text{otherwise};
\end{cases}
\een
when $n-b-1 \leq l -1 $, i.e. $b + l \geq n$,
\ben
K(b, l, j) = \begin{cases}
n, & b+l-n+1 \leq j \leq b, \\
0, & \text{otherwise}.
\end{cases}
\een
\end{proof}

\subsection{Crepant Resolution Conjecture for type A surface singularities in all genera}

For $k \geq 0$ and $1 \leq a \leq n-1$,
define
\ben
&& y_{a,k} = \frac{2i}{n} \sum_{b=1}^{n-1} \sin \frac{b\pi}{n} \cdot \xi_n^{ab} x_{b,k} u_b.
\een
For $k \geq 0$, $1 \leq s \leq t \leq n-1$, define
\ben
y_{s\to t, k} = \sum_{a=s}^t y_{a,k}.
\een

\begin{theorem}
Up to polynomial terms of degree $\leq 3$ in $y_{a, k}$,
the modified potential function
$\tilde{F}_g^{[\bC^2/\bZ_n]}(\{x_{1,k}, \dots, x_{n-1, k}\}_{k \geq 0}; u; z)$ is equal to
\ben
&& (-1)^{g} 2t \sum_{d=1}^{\infty} d^{2g-3}\sum_{1 \leq s \leq t \leq n-1} \biggl(\xi_n^{t-s+1}
\exp \biggl( \sum_{k \geq 0} y_{s\to t, k} \frac{z^k}{4^k\cdot (2k+1)!!} \biggr) \biggr)^d \\
& = & (-1)^{g} 2t \sum_{1 \leq s \leq t \leq n-1} \Li_{3-2g}\biggl(\xi_n^{t-s+1}
\exp \biggl( \sum_{k \geq 0} y_{s\to t, k} \frac{z^k}{4^k\cdot (2k+1)!!} \biggr) \biggr)
\een
after analytic continuations.
\end{theorem}

\begin{proof}
Combining (\ref{eqn:F}) with Lemma \ref{lm:Combo} one gets:
\ben
&& \tilde{F}_g^{[\bC^2/\bZ_n]}( \{x_{i,k}\}_{1 \leq i \leq n-1, k \geq 0};u;z)   \\
& = & (-1)^{g-1} t n^{2g-3} \sum_{m \geq 0}
\biggl( \sum_{\substack{0 \leq b \leq n-1, 1 \leq l \leq n-1\\ b+l < n}}
\big(n\sum_{k\geq 0} y_{b+1\to b+l, k} \frac{z^k}{4^k\cdot (2k+1)!!} \big)^m \\
& + & \sum_{\substack{0 \leq b \leq n-1, 1 \leq l \leq n-1\\ b+l > n}}
\big(-n\sum_{k\geq 0} y_{b+l-n+1 \to b, k}\frac{z^k}{4^k\cdot (2k+1)!!}\big)^m \biggr) \\
&& \cdot \sum_{c=0}^{n-1} \xi_n^{lc} \frac{B_{2g-2+m}(c/n)}{(2g-2+m) \cdot m!} \\
& = & (-1)^{g-1} 2t \sum_m n^{2g-3+m}
\sum_{1 \leq s \leq t \leq n-1}  \biggl(\sum_{k \geq 0} y_{s\to t, k}
\frac{z^k}{4^k\cdot (2k+1)!!}\biggr)^m \\
&& \cdot \sum_{c=0}^{n-1} \xi_n^{c(t-s+1)} \frac{B_{2g-2+m}(c/n)}{(2g-2+m) \cdot m!},
\een
where in the second equality we have used:
\be
B_n(1-x) = (-1)^n B_n(x).
\ee

On the other hand,
for $0 < l < n$ and $u < 0$,
\ben
&& 1 + \sum_{d =1}^{\infty} (\xi_n^l e^u)^d
= \frac{1}{1 - \xi_n^l e^u}
= - \frac{\sum_{c=0}^{n-1} \xi_n^{lc} e^{cu}}{e^{nu}- 1}
= - \sum_{c=0}^{n-1} \xi_n^{lc} \sum_{k=1}^{\infty} \frac{B_k(c/n)}{k!} (nu)^{k-1}.
\een
Integrating three times:
\be \label{eqn:Int1}
u + \sum_{d=1}^{\infty} \frac{(\xi_n^l e^u)^d}{d} - \Li_1(\xi_n^l)
= - \sum_{c=0}^{n-1} \xi_n^{lc} \sum_{m=1}^{\infty} \frac{B_m(c/n)}{m \cdot m!} n^{m-1} u^m,
\ee
\ben
\frac{u^2}{2} + \sum_{d=1}^{\infty} \frac{(\xi_n^le^u)^d}{d^2}
- \Li_2(\xi_n^l) - \Li_1(\xi_n^l) u
= - \sum_{c=0}^{n-1} \xi_n^{lc} \sum_{m=2}^{\infty} \frac{B_{m-1}(c/n)}{(m-1) \cdot m!} n^{m-2} u^m,
\een
\be \label{eqn:Int3}
\begin{split}
& \frac{u^3}{6} + \sum_{d=1}^{\infty} \frac{(\xi_n^le^u)^d}{d^3}
- \Li_3(\xi_n^l) - \Li_2(\xi_n^l) u
- \half \Li_1(\xi_n^l) u^2 \\
= & - \sum_{c=0}^{n-1} \xi_n^{lc} \sum_{m=3}^{\infty}
\frac{B_{m-2}(c/n)}{(m-2) \cdot m!} n^{m-3} u^{m}.
\end{split}
\ee
Also, by differentiating $2g-3$ (for $g > 1$) times:
\be \label{eqn:Int2g-3}
\sum_{d=1}^\infty d^{2g-3} (\xi_n^le^u)^d
= - \sum_{c=0}^{n-1} \xi_n^{lc} \sum_{m=0}^{\infty} \frac{B_{2g-2+m}(c/n)}{(2g-2+m) \cdot m!} n^{2g-3+m}
u^m.
\ee
Therefore
up to polynomial terms of degree $\leq 3$,
$\tilde{F}_g^{[\bC^2/\bZ_n]}(\{x_{i,k} \}_{1 \leq i \leq n-1,k \geq 0};u;z) $ is equal to
\ben
&& (-1)^{g} 2t \sum_{d=1}^{\infty} d^{2g-3}\sum_{1 \leq s \leq t \leq n-1} \biggl(\xi_n^{t-s+1}
\exp \biggl( \sum_{k \geq 0} y_{s\to t, k} \frac{z^k}{4^k\cdot (2k+1)!!} \biggr) \biggr)^d \\
& = & (-1)^{g} 2t \sum_{1 \leq s \leq t \leq n-1} \Li_{3-2g} \biggl(\xi_n^{t-s+1}
\exp \biggl( \sum_{k \geq 0} y_{s\to t, k} \frac{z^k}{4^k\cdot (2k+1)!!} \biggr)\biggr)
\een
after analytic continuations.
This completes the proof.
\end{proof}

Hence by taking the coefficient of $z^g$ we get:

\begin{theorem} \label{thm:Main2}
Up to polynomial terms of degree $\leq 3$ in $y_{a, k}$,
the stationary potential function $F_g^{[\bC^2/\bZ_n]}(\{x_{i,k}\}_{1 \leq i \leq n-1, k \geq 0} ; u)$
of the equivariant orbifold Gromov-Witten invariants of $[\bC^2/\bZ_n]$ is equal to
\ben
&& (-1)^{g} 2t \sum_{d=1}^{\infty} d^{2g-3}\sum_{1 \leq s \leq t \leq n-1} \xi_n^{(t-s+1)d}
\sum_{\sum_{k \geq 0} k m_k = g} \prod_{k \geq 0}
\frac{d^{m_k}y_{s\to t, k}^{m_k}}{m_k! 4^{km_k} \cdot [(2k+1)!!]^{m_k}}
\een
after analytic continuations,
where $y_{s \to t, k} = \sum_{a=s}^t y_{a,k}$.
\end{theorem}

Denote by $\pi: \widehat{\bC^2/\bZ_n} \to \bC^2/\bZ_n$ the minimal resolution of
$\bC^2/\bZ_n$.
Let $E_1$, $\dots$, $E_{n-1}$ be the exceptional divisors.
Let $(g_{ij}) = (E_i \cdot E_j)$ be the intersection matrix,
and let $(g^{ij})$ be the inverse matrix of $(g_{ij})$.
Then $\{C^i = g^{ij}E_j: \;i=1, \dots, n-1\}$ is the dual basis to the basis
$E_1, \dots, E_{n-1}$.
By \cite{Gon-Ver},
for each nontrivial representation $V_i$ of $\bZ_n$ on which $\xi_n$ acts as
multiplication by $e^{2\pi i \sqrt{-1}/n}$,
there is a holomorphic line bundle $L_i$ on $\widehat{\bC^2/\bZ_n}$
whose first Chern class is $C^i$.
One can use virtual localization to define
the Gromov-Witten invariants of $\widehat{\bC^2/\bZ_n}$.
By \cite[Theorem 1.1]{Mau},
for curve classes of the form $\beta = d (E_a+E_{i+1} + \cdots + E_b)$, $1 \leq a \leq b \leq n-1$,
and consider integers $a \leq l_1, \dots, l_m \leq b$ and $k_1, \dots, k_m \geq 0$
such that $k_1 + \cdots + k_m = g$,
\ben
&& \cor{ \prod_{i=1}^m \tau_{k_i}(C^{l_i})}_{g, \beta}^{\widehat{\bC^2/\bZ_n}}
= (-1)^g 2t d^{2g-3 +m} \prod_{i=1}^m \frac{1}{4^{k_i} (2k_i+1)!!}.
\een
Other correlators vanish.
Therefore,
the instanton part of stationary partition function of $\widehat{\bC^2/\bZ_n}$ is
\ben
&& \sum_m \frac{1}{m!} \sum_{\beta \in H_2(\widehat{\bC^2/\bZ_n}; \bZ)-\{0\}}
\sum_{\sum_{i=1}^m k_i= g} \sum_{l_1, \dots, l_m=1}^{n-1}
\cor{ \prod_{i=1}^m \tau_{k_i}(C^{l_i})}_{g, \beta}^{\widehat{\bC^2/\bZ_n}} q^{\beta}
\prod_{i=1}^m y_{l_i, k_i} \\
& = & \sum_m \frac{1}{m!} \sum_{d=1}^{\infty} \sum_{1 \leq a \leq b \leq n-1}
\sum_{k_1 + \cdots + k_m = g} \sum_{a \leq l_1, \dots, l_m \leq b}
(-1)^g 2t d^{2g-3 +m} \\
&& \cdot \prod_{i=1}^m \frac{1}{4^{k_i} (2k_i+1)!!}
\cdot \prod_{i=a}^b q_i^d \cdot \prod_{i=1}^m y_{l_i, k_i} \\
& = & (-1)^g2t \cdot \sum_{d=1}^{\infty}d^{2g-3} \sum_{1 \leq a \leq b \leq n-1}  \prod_{i=a}^b q_i^d
\sum_{\sum_{k \geq 0} k m_k = g} \prod_{k \geq 0}
\frac{d^{m_k}y_{a\to t, b}^{m_k}}{m_k! 4^{km_k} \cdot [(2k+1)!!]^{m_k}}.
\een
Hence we have established the Crepant Resolution Conjecture in all genera for type A
resolutions for the stationary part of the partition functions,
because one now only has to take $q_i = e^{2\pi \sqrt{-1}/n}$.

\section{Crepant Resolution Conjecture for $[\bC^2/\bZ_n] \times \bC$}

In this section we establish a version of the Crepant Resolution Conjecture for
$[\bC^2/\bZ_n] \times \bC$ in all genera.

\subsection{The equivariant Gromov-Witten invariants of $[\bC^2/\bZ_n] \times \bC$}
\label{sec:3D}

\subsubsection{The circle actions on $[\bC^2/\bZ_n] \times \bC$}
Fix an integer $a \in \bZ$,
consider the following circle action $\bC^3$:
\be \label{eqn:SAction}
e^{i\theta} \cdot (z_1, z_2, z_3) = (e^{ia\theta} z_1, e^{-i(a+1)\theta}z_2, e^{i\theta}z_3).
\ee
This action commutes with the following $\bZ_n$-action:
$$\xi_n \cdot (z_1, z_2, z_3)
= (\xi_n \cdot z_1, \xi_n^{-1} \cdot z_2, z_3).$$
Hence we get an induced circle action on the orbifold $[\bC^2/\bZ_n] \times \bC$.
Note both the circle action and the $\bZ_n$-action preserve the holomorphic volume form
$dz_1 \wedge dz_2 \wedge dz_3$ on $\bC^3$.

\subsubsection{Definition of the equivariant Gromov-Witten invariants}
As in the $[\bC^2/\bZ_n]$ case,
one can define the Gromov-Witten invariants of $[\bC^2/\bZ_n] \times \bC$ equivariantly using
the above circle actions and virtual localizations.
It is straightforward to see that the equivariant correlators are given by
\ben
\cor{\prod_{j=1}^{m} \tau_0(e_{\cl{\omega^{a_j}}})}_g^{[\bC^2/\bZ_n]\times \bC} =  \int_{\Mbar_{g, m}(\cB \bZ_n; \cl{\omega^{a_1}}, \dots, \cl{\omega^{a_m}} )}
\frac{1}{t} c_t(\bF_0^1) c_{at}(\bF_1^1) c_{-(a+1)t}(\bF_{-1}^1) ,
\een
where $\bF_0^1$ is the vector bundle whose fiber at a twisted stable map $f: \Sigma \to \cB\bZ_n$
is $H^1(\Sigma, f^*V_0)$, where $V_0$ is the trivial representation of $\bZ_n$.

\begin{lemma} \label{lm:CC}
One has
\ben
&& \frac{1}{t}
\int_{\Mbar_{g, m}(\cB \bZ_n; \cl{\omega^{a_1}}, \dots, \cl{\omega^{a_m}} )}
c_t(\bF_0)c_{at}(\bF_1) c_{-(a+1)t}(\bF_{-1}) \\
& = & (-1)^{\sum_i a_i/n - 1} (2g-3+m)!
\int_{\Mbar_{g, m}(\cB \bZ_n; \cl{\omega^{a_1}}, \dots, \cl{\omega^{a_m}} )}
\lambda_{g,0} \cdot \ch_{2g-3+m}(\bF_1),
\een
where $\lambda_{j,0} = (-1)^j c_j(\bF_0)$.
\end{lemma}

\begin{proof}
By (\ref{eqn:Chern})  we have
\ben
&& \frac{1}{t}
\int_{\Mbar_{g, m}(\cB \bZ_n; \cl{\omega^{a_1}}, \dots, \cl{\omega^{a_m}} )}
c_t(\bF_0)c_{at}(\bF_1) c_{-(a+1)t}(\bF_{-1}) \\
& = & \frac{1}{t} \int_{\Mbar_{g, m}(\cB \bZ_n; \cl{\omega^{a_1}}, \dots, \cl{\omega^{a_m}} )}
( t^g + t^{g-1} c_1(\bF_0) + \cdots + c_g(\bF_0)) \\
&&  \cdot ( (at)^{r_1} + (at)^{r_1-1} c_1(\bF_1) + \cdots + c_{r_1}(\bF_1) \\
&& \cdot ( [-(a+1)t]^{\bar{r}_1} + [-(a+1)t]^{\bar{r}_1-1}c_1(\bF_{-1}) + \cdots + c_{\bar{r}_1}(\bF_{-1})) \\
& = & \frac{1}{t} \int_{\Mbar_{g, m}(\cB \bZ_n; \cl{\omega^{a_1}}, \dots, \cl{\omega^{a_m}} )}
( t^g + t^{g-1} c_1(\bF_0) + \cdots + c_g(\bF_0)) \\
&&  \cdot (-t) (-1)^{(g-1+\sum_i a_i/n)-1} ( (2g-3+m)! \ch_{2g-3+m}(\bF_1) + \cdots ) \\
& = & (-1)^{\sum_i a_i/n - 1} (2g-3+m)!
\int_{\Mbar_{g, m}(\cB \bZ_n; \cl{\omega^{a_1}}, \dots, \cl{\omega^{a_m}} )}
\lambda_{g,0} \cdot \ch_{2g-3+m}(\bF_1).
\een
\end{proof}

\begin{proposition} \label{prop:3D}
For $m \geq 2$ and $g \geq 0$,
we have
\be \label{eqn:3DLambda}
\begin{split}
& \int_{\Mbar_{g, m}(\cB \bZ_n; \cl{\omega^{a_1}}, \dots, \cl{\omega^{a_m}})}
\lambda_{g,0} \ch_{2g-3+m}(\bF_1) \\
= & - \half n^{2g-2} k_g  \sum_{l=0}^{n-1}
\prod_{i=1}^m (\xi_n^{a_il}-1)
\cdot \sum_{a=0}^{n-1} \xi_n^{al} \frac{B_{2g-2+m}(a/n)}{(2g-2+m)!},
\end{split}
\ee
where
\be \label{eqn:Bg}
k_g = \sum_{g_1+g_2=g} b_{g_1}  b_{g_2}, \;\;\;\;
\sum_{g \geq 0} b_g \lambda^{2g}
= \sum_{g \geq 0} \lambda^{2g} \int_{\Mbar_{g,1}} \lambda_g\psi_1^{2g-2}
= \frac{\lambda/2}{\sin (\lambda/2)}.
\ee

\end{proposition}

\begin{proof}
By Tseng's GRR relations for Hurwitz-Hodge integrals \cite{Tse}:
\ben
&& \int_{\Mbar_{g, m}(\cB \bZ_n; \cl{\omega^{a_1}}, \dots, \cl{\omega^{a_m}})}
\lambda_{g,0} \ch_{2g-3+m}(\bF_1) \\
& = & - \frac{B_{2g-2+m}}{(2g-2+m)!}
\int_{\Mbar_{g, m+1}(\cB \bZ_n; \cl{\omega^{a_1}}, \dots, \cl{\omega^{a_m}}, \cl1)}
\lambda_{g,0} \bpsi_{m+1}^{2g-2+m} \\
& + & \sum_{i=1}^m \frac{B_{2g-2+m}(a_i/n)}{(2g-2+m)!}
\int_{\Mbar_{g, m}(\cB \bZ_n; \cl{\omega^{a_1}}, \dots, \cl{\omega^{a_m}})} \lambda_{g,0}
\bpsi_i^{2g-3+m} \\
& - & \half \sum_{\substack{I \coprod J = [m] \\ g_1+g_2=g}}\ '
\frac{B_{2g-2+m}(c(a_I)/n)}{(2g-2+m)!} \\
&& \cdot \sum_{l=0}^{r_1+\bar{r}_1-2} (-1)^l
\int_{\Mbar_{g_1, \sum_k j_k +1}(\cB\bZ_n; \coprod_{i\in I} \cl{\omega}^{a_i},
\cl{\omega^{-\sum_{i \in I} a_i}})} \lambda_{g_1,0} \bpsi_{|I|+1}^{r_1+\bar{r}_1-2-l} \\
&& \cdot n \cdot
\int_{\Mbar_{g_2,|J|+ 1}(\cB\bZ_n; \coprod_{j \in J} \cl{\omega^{a_j}}, \cl{\omega^{-\sum_{j \in J} a_j}})}
\lambda_{g_2,0} \bpsi_{|J|+1}^{l}.
\een
Here the prime sign in the summation $\sum'_{\substack{I \coprod J = [m] \\ g_1 + g_2 = g}}$
in the third term on the right-hand side
means the following stability conditions are satisfied:
$$2g_1 - 1 + |I| >0, \;\;\; 2g_2 - 1 + |J| > 0.$$
Our convention is that $\lambda_{0,0} = 1$.
Now we use the morphism
$$\varphi: \Mbar_{g_1, |J|+1}(\cB \bZ_n; \coprod_{i \in J} \cl{\omega^{a_i}},
\cl{\omega^{-\sum_{i \in J}  a_i}})
\to \Mbar_{g, \sum_i m_i+1}$$
that forgets the orbifold structure.
Because $\lambda_{g_1,0} = \varphi^*\lambda_{g_1}$,
$\bpsi_{\sum_i j_i +1} = \varphi^*\psi_{\sum_i j_i+1}$
and $\varphi$ is  of degree $n^{2g_1-1}$,
we have for $2g_2 - 2 + |J| > 0$:

\ben
&& \int_{\Mbar_{g_2, |J|+1}(\cB \bZ_n; \coprod_{i \in J} \cl{\omega^{a_i}}, \cl{\omega^{-\sum_{i \in J} a_i}})}
\lambda_{g_2,0} \bpsi_{|J|+1}^l \\
& = & n^{2g_2-1} \delta_{l, 2g_2-2+|J|} \int_{\Mbar_{g_2, |J|+1}}
\lambda_{g_2} \psi_{|J|+1}^{2g_2-2+|J|} \\
& = &
n^{2g_2-1} \delta_{l, 2g_2-2+|J|} b_{g_2}.
\een
For $2g_2 - 2 + |J| \leq 0$,
we have used the conventions (\ref{eqn:Conv1}) and (\ref{eqn:Conv2}) to understand the following two cases:
\ben
&& \int_{\Mbar_{g_2, |J|+1}(\cB \bZ_n; \prod_{i \in J} \cl{\omega^{a_i}}, \cl{\omega^{-\sum_{i \in J} a_i}})}
\lambda_{g_2,0} \bpsi_{|J|+1}^l \\
& = & \begin{cases}
\delta_{l,-2}, & g_2 = 0, J = \emptyset, \\
-\delta_{l,-1}, & g_2 = 0, |J|= 1.
\end{cases}
\een
Therefore,
one finds
\ben
&& \int_{\Mbar_{g, m}(\cB \bZ_n; \cl{\omega^{a_1}}, \dots, \cl{\omega^{a_m}})}
\lambda_{g,0} \ch_{r_1+\bar{r}_1-1}(\bF_1) \\
& = & - \half n^{2g-1} \sum_{I \coprod J} \sum_{g_1+g_2 = g} (-1)^{|J|}
\frac{B_{2g-2+m}(c(a_I)/n)}{(2g-2+m)!} b_{g_1} b_{g_2} \\
& = & - \frac{k_g}{2} n^{2g-1}  \sum_{I \coprod J} (-1)^{|J|}
\frac{B_{2g-2+m}(c(a_I)/n)}{(2g-2+m)!} \\
& = & - \half n^{2g-2} k_g  \sum_{l=0}^{n-1}
\prod_{i=1}^m (\xi_n^{a_il}-1)
\cdot \sum_{c=0}^{n-1} \xi_n^{lc} \frac{B_{2g-2+m}(c/n)}{(2g-2+m)!}.
\een
\end{proof}

\subsection{Crepant Resolution Conjecture for $[\bC^2/\bZ_n] \times \bC$}

Define the instanton part of
genus $g$ equivariant orbifold Gromov-Witten potential function by:
\ben
&& F_g^{[\bC^2/\bZ_n] \times \bC}(u_{1}, \dots, u_{n-1}) \\
& = & \sum_{m \geq 1} \frac{1}{m!} \sum_{\substack{1 \leq a_1, \dots, a_m \leq n-1 \\ \sum_j a_j \equiv 0 \pmod{n}}}
\cor{ \prod_{j=1}^{m} \tau_0(e_{\cl{\omega^{a_j}}} )}_g^{[\bC^2/\bZ_n] \times \bC}
 \cdot \prod_{j=1}^m u_{a_j}.
\een
Define
\ben
&& F^{[\bC^2/\bZ_n]\times \bC}(\lambda;u_1, \dots, u_{n-1})
= \sum_{g \geq 0} \lambda^{2g-2} F_g^{[\bC^2/\bZ_n]\times \bC}(u_1, \dots, u_{n-1}).
\een

\begin{theorem}
Up to polynomial terms of degree $\leq 3$ in $u_1, \dots, u_{n-1}$,
we have
\be
F^{[\bC^2/\bZ_n]\times \bC}(\lambda;u_1, \dots, u_{n-1}) = \sum_{d \geq 1} \frac{1}{4 d \sin^2(d\lambda/2)} \sum_{1 \leq s \leq t \leq n-1} \biggl(\xi_n^{t-s+1}
e^{v_{s\to t}} \biggr)^d,
\ee
where
\ben
v_{s\to t} = \sum_{a=s}^t v_{a}, \;\;\;\;
v_j = \frac{i}{n} \sum_{k=1}^{n-1} \sqrt{2 - 2 \cos \frac{2 k \pi}{n}} \xi_n^{jk} u_k.
\een
\end{theorem}

\begin{proof}
By Lemma \ref{lm:CC} and Proposition \ref{prop:3D},
we have
\ben
&& F^{[\bC^2/\bZ_n] \times\bC}_g(u_1, \dots, u_{n-1}) \\
& = &  \half n^{2g-2} k_g
\sum_{m \geq 1} \frac{1}{m!} \sum_{\substack{1 \leq a_1, \dots, a_m \leq n-1 \\
\sum_{i=1}^m a_i \equiv 0 \pmod{n}}} (-1)^{\sum_i a_i/n}
\sum_{l=1}^{n-1}
\prod_{i=1}^m [(\xi_n^{a_il}-1) u_{a_i} ] \\
&& \cdot \sum_{c=0}^{n-1} \xi_n^{lc} \frac{B_{2g-2+m}(c/n)}{2g-2+m} \\
& = & \half n^{2g-3} k_g
\sum_{m \geq 1} \frac{1}{m!} \sum_{a_1, \dots, a_m=1}^{n-1} \sum_{b=0}^{n-1} \xi_n^{b \sum_i a_i} \xi_{2n}^{\sum_i a_i}
\sum_{l=1}^{n-1}
\prod_{i=1}^m [(\xi_n^{a_il}-1) u_{a_i} ] \\
&& \cdot \sum_{c=0}^{n-1} \xi_n^{lc} \frac{B_{2g-2+m}(c/n)}{2g-2+m} \\
& = & \half n^{2g-3} k_g \sum_{b=0}^{n-1} \sum_{l=1}^{n-1} \sum_{m \geq 1} \frac{1}{m!} \biggl(
\sum_{1 \leq a \leq n-1}  \xi_{2n}^a   \xi_n^{ba}
 (\xi_n^{la} - 1) u_a \biggr)^m \\
&& \cdot \sum_{c=0}^{n-1} \xi_n^{lc} \cdot \frac{B_{2g-2+m}(c/n)}{2g-2+m}.
\een
By Lemma \ref{lm:Combo},
\ben
&& F^{[\bC^2/\bZ_n] \times\bC}_g(u_1, \dots, u_{n-1}) \\
& = & \half n^{2g-3} k_g  \sum_{m \geq 1}
\biggl( \sum_{\substack{0 \leq b \leq n-1, 1 \leq l \leq n-1\\ b+l < n}}
\big(n v_{b+1\to b+l} \big)^m \\
& + & \sum_{\substack{0 \leq b \leq n-1, 1 \leq l \leq n-1\\ b+l > n}}
\big(-n v_{b+l-n+1 \to b} \big)^m \biggr)
\cdot \sum_{c=0}^{n-1} \xi_n^{lc} \frac{B_{2g-2+m}(c/n)}{(2g-2+m) \cdot m!} \\
& = & k_g \sum_{m \geq 1} n^{2g-3+m}
\sum_{1 \leq s \leq t \leq n-1}  v_{s\to t}^m
\cdot \sum_{c=0}^{n-1} \xi_n^{c(t-s+1)} \frac{B_{2g-2+m}(c/n)}{(2g-2+m) \cdot m!},
\een
Therefore,
by (\ref{eqn:Int1}), (\ref{eqn:Int3}), (\ref{eqn:Int2g-3}),
up to polynomials terms of degree $\leq 3$,
the potential function $F_g^{[\bC^2/\bZ_n] \times \bC}(u_{1}, \dots, u_{n-1}) $ is equal to
\ben
&& k_g \sum_{d=1}^{\infty} d^{2g-3}\sum_{1 \leq s \leq t \leq n-1} \biggl(\xi_n^{t-s+1}
e^{v_{s\to t}} \biggr)^d
\een
after analytic continuations.
The proof is completed by (\ref{eqn:Bg}).
\end{proof}

The potential function of $\widehat{\bC^2/\bZ_n} \times \bC$ can be defined and computed also by virtual
localization.
For $\beta = d_1E_1 + \cdots + d_{n-1}E_{n-1} \neq 0$,
\ben
\cor{1}^{\widehat{\bC^2/\bZ_n} \times \bC}_{g, \beta}
= \int_{\Mbar_{g,0}(\widehat{\bC^2/\bZ_n} \times \bC;\beta)^{virt}_T} 1.
\een
Introduce  degree tracking variable $Q_1, \dots, Q_{n-1}$.
We define
\ben
&&  F_g^{\widehat{\bC^2/\bZ_n}\times \bC}(Q_1, \dots, Q_{n-1})
= \sum_{\beta \neq 0}
\cor{1}^{\widehat{\bC^2/\bZ_n} \times \bC}_{g, \beta} Q^{\beta},
\een
where for $\beta = d_1 E_1 + \cdots + d_{n-1} E_{n-1}$,
$Q^{\beta} =Q_1^{d_1} \cdots Q_{n-1}^{d_{n-1}}$.
Define
\ben
&& F^{\widehat{\bC^2/\bZ_n}\times \bC}(\lambda;Q_1, \dots, Q_{n-1})
= \sum_{g \geq 0} \lambda^{2g-2}  F_g^{\widehat{\bC^2/\bZ_n}\times \bC}(Q_1, \dots, Q_{n-1}),
\een
where $\lambda$ is the genus tracking variable.

\begin{theorem} \label{thm:F3D}
We have
\be \label{eqn:F3D}
F^{\widehat{\bC^2/\bZ_n}\times \bC}(\lambda;Q_1, \dots, Q_{n-1})
= \sum_{1 \leq a \leq b \leq n-1} \sum_{d=1}^{\infty} \frac{\prod_{k=a}^b Q_k^d}{d}
\frac{1}{4 \sin^2 (d\lambda/2)}.
\ee
\end{theorem}

\begin{proof}
The circle action (\ref{eqn:SAction})
induces a circle action on $\widehat{\bC^2/\bZ_n} \times \bC$,
with fixed points $p_1, \dots, p_n$.
The tangents weights at $p_i$ are $(na+i-1)t$, $-(na+i)t$, and $t$.
By virtual localization one encounters two-partition Hodge integrals
which have been studied in \cite{Zho-1, Liu-Liu-Zho}.
By the method of \cite{Zho0},
we get the following expression for the potential function:
\ben
&& F^{\widehat{\bC^2/\bZ_n}\times \bC}(\lambda;Q_1, \dots, Q_{n-1})
= \log \sum_{\mu^1, \dots, \mu^{n-1}}
\prod_{i=1}^{n} (\cW_{\mu^{i-1},\mu^i}(q) q^{-\kappa_{\mu^i}} Q_i^{|\mu^i|}),
\een
where $\mu^0 = \mu^n = \emptyset$, $q=e^{\sqrt{-1}\lambda}$.
See notations see \cite{Zho0}.
One can rewrite this as in \cite{Zho1} by Schur calculus.
Indeed,
we have \cite{Zho-1}:
\be
\cW_{\mu,\nu}(q) = (-1)^{|\mu|+|\nu|}q^{(\kappa_{\mu} + \kappa_{\nu})/2}
\sum_{\eta} s_{\mu/\eta}(q^{-\rho}) \cdot s_{\nu/\eta}(q^{-\rho}),
\ee
where $q^{-\rho} = (q^{1/2}, q^{3/2}, \dots)$.
Hence
\ben
&& F^{\widehat{\bC^2/\bZ_n}\times \bC}(\lambda;Q_1, \dots, Q_{n-1}) \\
& = & \log \sum_{\mu^1, \dots, \mu^{n-1}} \sum_{\eta^1, \dots, \eta^{n-2}}
\prod_{i=1}^{n-1} s_{\mu^i/\eta^{i-1}}(q^{-\rho}) Q_i^{|\mu^i|} s_{\mu^i/\eta^i}(q^{-\rho}).
\een
Therefore,
by \cite[Lemma 3.1]{Zho1},
\ben
&& F^{\widehat{\bC^2/\bZ_n}\times \bC}(\lambda;Q_1, \dots, Q_{n-1}) \\
& = & \log \prod_{1 \leq a \leq b \leq n-1} \prod_{i,j = 1}^{\infty}
(1 - \prod_{k=a}^b Q_k \cdot q^{i+j-1}) \\
& = & - \sum_{1 \leq a \leq b \leq n-1} \sum_{d=1}^{\infty} \frac{\prod_{k=a}^b Q_k^d}{d}
\sum_{m=1}^{\infty} mq^{md} \\
& = & - \sum_{1 \leq a \leq b \leq n-1} \sum_{d=1}^{\infty} \frac{\prod_{k=a}^b Q_k^d}{d}
\frac{q^d}{(1-q^d)^2} \\
& = &  \sum_{1 \leq a \leq b \leq n-1} \sum_{d=1}^{\infty} \frac{\prod_{k=a}^b Q_k^d}{d}
\frac{1}{4 \sin^2 (d\lambda/2)}.
\een
\end{proof}

By combining the above two Theorems,
we get

\begin{theorem} \label{thm:CRC3D}
Up to polynomial terms of degree $\leq 3$ in $u_1, \dots, u_{n-1}$,
we have
\be
F^{[\bC^2/\bZ_n]\times \bC}(\lambda;u_1, \dots, u_{n-1}) =
F^{\widehat{\bC^2/\bZ_n}\times \bC}(\lambda;Q_1, \dots, Q_{n-1})
\ee
after analytic continuation, where
\ben
Q_j = \xi_n e^{v_j}, \;\;\;\;
v_j = \frac{\sqrt{-1}}{n} \sum_{k=1}^{n-1} \sqrt{2 - 2 \cos \frac{2 k \pi}{n}} \xi_n^{jk} u_k.
\een
\end{theorem}

{\em Acknowledgements}.
The author thanks Professors Jim Bryan, Yongbin Ruan and Hsian-Hua Tseng for their interest in this work.
In an earlier version of this paper Proposition 2.1 was stated as a conjecture.
Without their encouragements and suggestions this paper will not appear in the present form.
This research is partially supported by two NSFC grants (10425101 and 10631050)
and a 973 project grant NKBRPC (2006cB805905).

\end{document}